\documentclass[12pt]{amsart}
\usepackage[a4paper, left=28mm, right=28mm, top=28mm, bottom=34mm]{geometry}
\usepackage[all]{xy}
\usepackage{comment}
\usepackage{amsmath}
\usepackage{amssymb}
\usepackage{amsthm}
\usepackage{mathrsfs}
\usepackage{mathtools}
\usepackage{mathabx}
\usepackage{scalefnt}
\usepackage{url}
\usepackage{comment}
\usepackage{pb-diagram}
\theoremstyle{plain}
\newtheorem{thm}{Theorem}[section]
\newtheorem{lem}[thm]{Lemma}
\newtheorem{cor}[thm]{Corollary}
\newtheorem{prop}[thm]{Proposition}

\theoremstyle{definition}

\newtheorem{rem}[thm]{Remark}

\numberwithin{equation}{section}

\def\F{{\mathbb F}}
\def\Q{{\mathbb Q}}
\def\R{{\mathbb R}}
\def\Z{{\mathbb Z}}
\def\C{{\mathbb C}}
\def\P{{\mathbb P}}

\def\GL{\mathop{\mathrm{GL}}\nolimits}

\def\Tr{\mathop{\mathrm{Tr}}\nolimits}

\def\dim{\mathop{\mathrm{dim}}\nolimits}
\def\div{\mathop{\mathrm{div}}\nolimits}

\def\Div{\mathop{\mathrm{Div}}\nolimits}

\def\H{\mathscr{H}}

\def\M{\mathscr{M}}
\def\F{\mathscr{F}}

\def\OO{\mathscr{O}}

\def\D{\mathscr{D}}

\def\l{\langle}
\def\r{\rangle}

\def\O{\mathop{{\null}\mathrm{O}}\nolimits}

\def\Supp{\mathop{\mathrm{Supp}}\nolimits}

\def\U{\mathrm{U}}
\def\S{\mathscr{S}}

\newcommand{\defeq}{\vcentcolon=}

\begin{document}

\title[Uniruledness of ball quotients]
{Uniruledness of some low-dimensional ball quotients}

\author{Yota Maeda}
\address{Department of Mathematics, Faculty of Science, Kyoto University, Kyoto 606-8502, Japan / Advanced Research Laboratory, Technology Infrastructure Center, Technology Platform, Sony Group Corporation, 1-7-1 Konan, Minato-ku}
\email{y.maeda.math@gmail.com}



\date{\today}

\begin{abstract}
We define reflective modular forms on complex balls and use a method of Gritsenko and Hulek to show that some ball quotients of dimensions 3, 4 and 5 are uniruled. 
We give examples of Hermitian lattices over the rings of integers of imaginary quadratic fields $\mathbb{Q}(\sqrt{-1})$ and $\mathbb{Q}(\sqrt{-2})$ for which the associated ball quotients are uniruled.
Our examples include the moduli space of 8 points on $\mathbb{P}^1$.
Moreover, we find that some of their Satake-Baily-Borel compactifications are rationally chain connected modulo certain cusps.
\end{abstract}

\maketitle

\section{Introduction}
Gritsenko and Hulek showed that some orthogonal modular varieties are uniruled \cite{Uniruledness}.
They used reflective modular forms \cite{Reflective} and the numerical criterion of uniruledness due to Miyaoka and Mori \cite{MM}.
In this paper, we apply this method to ball quotients associated with Hermitian forms of signatures $(1,3)$, $(1,4)$, and $(1,5)$, and show that some of them are uniruled.
Moreover, we prove that some of their Satake-Baily-Borel compactifications are rationally chain connected modulo certain divisors.

Let $F\defeq\Q(\sqrt{d})$ be an imaginary quadratic field
with its ring of integers $\OO_F$,
where $d < 0$ is a square-free negative integer.
Let $L$ be a free $\OO_F$-module of rank $n+1$
equipped with a Hermitian form $\l\ ,\ \r\colon L\times L\to F$ for some $n>1$ satisfying $\Tr_{F/\Q}\l x,x\r\in\Z$ for any $x\in L$.
A pair $(L,\l\ ,\ \r)$ is called a {\em Hermitian lattice}.
Here, Hermitian forms are supposed to be complex linear in the first argument and complex
conjugate linear in the second argument.
We fix an embedding $F \hookrightarrow \C$ and assume that $\l\ ,\ \r$ has signature $(1,n)$.

Let $\mathbf{U}(L)$ be the unitary group associated with $(L,\l\ ,\ \r)$.
It is a group scheme over $\Z$ defined by 
\[\mathbf{U}(L)(R)\defeq\{g\in\GL_R(L\otimes_{\Z}R)\mid \l gv,gw\r=\l v,w\r\ \mathrm{for\ any\ }v,w\in L\otimes_{\Z}R\}\]
for any $\Z$-algebra $R$.
Let us denote by $\U(L)\defeq\mathbf{U}(L)(\Z)$ below.
Let $D_L$ be the Hermitian symmetric domain associated with
 $\mathbf{U}(L)(\R)\cong\U(1,n)$, i.e.,
\[D_L\defeq\{[w]\in\P(L\otimes_{\OO_F}\C)\mid\l w,w\r >0 \}\]
which is biholomorphic to the $n$-dimensional complex ball.
For $\Gamma\subset\U(L)$, we put
\[ \F_L(\Gamma)\defeq \Gamma\backslash D_L. \]
It is called the {\em ball quotient}, a quasi-projective variety over $\C$.

Recall that an irreducible variety $X$ over $\C$ is called {\em uniruled} if there exists a dominant rational map $Y\times\P^1\dashrightarrow X$
where $Y$ is an irreducible variety over $\C$ with $\dim Y=\dim X-1$.
Uniruled varieties have Kodaira dimension $-\infty$.
The converse is not known in dimension greater than 3.

First, let us state the main theorem in this paper.
This is a criterion for when $\F_L(\Gamma)$ is uniruled, a unitary analogue of \cite[Theorem 2.1]{Uniruledness}.

\begin{thm}[Theorem \ref{uniruledness}]
\label{mainthm:uniruledness}
Let $F$ be an imaginary quadratic field and $(L,\l\ ,\ \r)$ be a Hermitian lattice over $\OO_F$ of signature $(1,n)$ for $n>1$.
Let $a,k > 0$ be positive integers satisfying $k > a(n+1)$.
If there exists a non-zero reflective modular form $F_{a,k}\in M_k(\Gamma,\chi)$ of weight $k$ for which the multiplicity of every irreducible component of $\div(F_{a,k})$ is less than or equal to $a$, then the ball quotient $\F_L(\Gamma')$ is uniruled for every arithmetic group $\Gamma'$ containing $\Gamma$.
\end{thm}
For the definition of reflective modular forms, see Section \ref{section:criterion} and Appendix \ref{Section:Reflective modular forms}.

We construct reflective modular forms on $D_L$ by using reflective modular forms on the type IV domain and we conclude that some ball quotients are uniruled.

Before stating applications, we introduce some invariants of quadratic lattices.
Let $(L_Q,(\ ,\ ))$ be the quadratic lattice of rank $2n+2$ over $\Z$ of signature $(2,2n)$, where $L_Q$ is $L$ considered as a free $\Z$-module and $(\ ,\ )\defeq\Tr_{F/\Q}\l\ ,\ \r$.
We say that $L_Q$ is {\em even} if $(v,v)\in 2\Z$ for all $v\in L_Q$.
We say that $L_Q$ is {\em $2$-elementary} if there exists a non-negative integer $\ell(L_Q)$ such that
\[L_Q^{\vee}/L_Q\cong (\Z/2\Z)^{\ell(L_Q)}.\]
Here $L_Q^{\vee}$ is the dual lattice of  $L_Q$, i.e.,
\[L_Q^{\vee}\defeq\{v\in L_Q\otimes_{\Z}\Q\mid (v,w)\in\Z\ \mathrm{for\ any\ }w\in L_Q\}.\]
We also define
\[\delta(L_Q)\defeq\begin{cases}0&\mathrm{if\ }(v,v)\in\Z\ \mathrm{for\ all}\ v\in L_Q^{\vee}\\
    1&\mathrm{if\ }(v,v)\not\in\Z\ \mathrm{for\ some}\ v\in L_Q^{\vee}).
    \end{cases}\]

We shall prove that certain ball quotients $X_L\defeq\F_L(\U(L))$ associated with Hermitian forms of signatures $(1,3)$, $(1,4)$ and $(1,5)$ are uniruled.
For $r\in L$,  the fractional ideal generated by 
\[\{\l\ell,r\r\mid\ell\in L\}\]
is denoted by $\Div(r)$.
The main applications of Theorem \ref{mainthm:uniruledness} are as follows.
\begin{thm}[Theorem \ref{Kodaira_dim_app}]
  \label{main_Yoshikawa}
  Assume that
  \begin{enumerate}
    \item $n=5$.
      \item $L_Q$ is even $2$-elementary, $\delta(L_Q)=0$, 
  and $\ell(L_Q)\leq 8$.
  Moreover, $\ell(L_Q)\leq 6$ if $F=\Q(\sqrt{-1})$ or $\Q(\sqrt{-3})$.
      \item  $2\cdot\Div(r)\subset\OO_F$ for any primitive $r\in L$ with $\l r,r\r=-1$.
  \end{enumerate}
  Then $X_L$ is uniruled.
\end{thm}
\begin{thm}[Theorem \ref{Kodaira_dim_app2}]
\label{main_Gritsenko}
    Assume that
  \begin{enumerate}
      \item $L_Q\cong\mathbb{U}\oplus\mathbb{U}(2)\oplus\mathbb{E}_8(-2)$.
      \item  $\Div(r)\subset\OO_F$ for any primitive $r\in L$ with $\l r,r\r=-2$ and $\Tr_{F/\Q}(\Div(r))\subset 2\Z$.
  \end{enumerate}
  Then  $X_L$ is uniruled.
\end{thm}
\begin{thm}[Theorem \ref{Kodaira_dim_app3}]
\label{main_Yoshikawa2}
    Assume that
  \begin{enumerate}
      \item $L_Q\cong\mathbb{U}\oplus\mathbb{U}\oplus\mathbb{D}_6$.
      \item  $2\cdot\Div(r)\subset\OO_F$ for any primitive $r\in L$ with $\l r,r\r=-1$.
  \end{enumerate}
  Then $X_L$ is uniruled.
\end{thm}
\begin{thm}[Theorem \ref{Kodaira_dim_app4}]
  \label{main_Yoshikawa3}
  Assume that
  \begin{enumerate}
      \item $L_Q\cong\mathbb{U}\oplus\mathbb{U}\oplus\mathbb{D}_4$, $\mathbb{U}\oplus\mathbb{U}(2)\oplus\mathbb{D}_4$ or $\mathbb{U}(2)\oplus\mathbb{U}(2)\oplus\mathbb{D}_4$.
      \item  $2\cdot\Div(r)\subset\OO_F$ for any primitive $r\in L$ with $\l r,r\r=-1$.
  \end{enumerate}
  Then $X_L$ is uniruled.
\end{thm}

\begin{rem}
More strongly, we can prove that the Satake-Baily-Borel compactification of $X_L$ is rationally chain connected modulo certain cusps; see Theorem \ref{rem:rationally_chain} and Remark \ref{rem:rationally_kondo}.
\end{rem}

In Table \ref{table:lattices_and_invariants}, we give examples of even $2$-elementary quadratic lattices over $\Z$ of signature $(2,10)$ satisfying $\delta(L_Q) = 0$.
Here, $\mathbb{U}$ is the hyperbolic lattice of signature $(1,1)$, and $\mathbb{E}_8$ and $\mathbb{D}_k$ are (positive definite) root lattices associated with the root systems $E_8$ and $D_k$,  respectively.

\vspace{0.2in}
\begin{table}[h]
\begin{center}
\renewcommand{\arraystretch}{1.5}
\caption{Lattices and invariants}
\label{table:lattices_and_invariants}
    \begin{tabular}{|c||c|c|} \hline
     Quadratic lattices $L_Q$&$\ell(L_Q)$&$\delta(L_Q)$\\ \hline
      $\mathbb{U}\oplus \mathbb{U}(2)\oplus \mathbb{E}_8(-2)$&10&0\\ \hline
      $\mathbb{U}\oplus \mathbb{U}\oplus \mathbb{E}_8(-2)$&8&0\\\hline
      $\mathbb{U}\oplus \mathbb{U}(2)\oplus \mathbb{D}_4(-1)\oplus \mathbb{D}_4(-1)$&6&0\\\hline
      $\mathbb{U}\oplus \mathbb{U}\oplus \mathbb{D}_4(-1)\oplus \mathbb{D}_4(-1)$&4&0\\\hline
      $\mathbb{U}\oplus \mathbb{U}\oplus \mathbb{D}_8(-1)$&2&0\\\hline
      $\mathbb{U}\oplus \mathbb{U}\oplus \mathbb{E}_8(-1)$&0&0\\\hline
    \end{tabular}
  \renewcommand{\arraystretch}{1.0}
  \end{center}
\end{table}
\vspace{0.2in}

\begin{rem}
If $L_Q\cong\mathbb{U}\oplus \mathbb{U}(2)\oplus \mathbb{D}_4(-1)\oplus \mathbb{D}_4(-1)$, then the ball quotient $X_L$ coincides with the moduli space of 8 points in $\P^1$ (see \cite{Kondo3}), which is thus uniruled by Theorem \ref{main_Yoshikawa}.
\end{rem}

\begin{cor}
\label{maincor:exist}
\begin{enumerate}
    \item   For $F=\Q(\sqrt{-1})$ or $\Q(\sqrt{-2})$, there exists a Hermitian lattice $L$ of signature $(1,3)$ over $\OO_F$ such that the associated ball quotient $X_L$ is uniruled.
    \item   For $F=\Q(\sqrt{-1})$, there exists a Hermitian lattice $L$ of signature $(1,4)$ over $\OO_F$ such that the associated ball quotient $X_L$ is uniruled.
    \item   For $F=\Q(\sqrt{-1})$ or $\Q(\sqrt{-2})$, there exists a Hermitian lattice $L$ of signature $(1,5)$ over $\OO_F$ such that the associated ball quotient $X_L$ is uniruled.
\end{enumerate}
\end{cor}
(See Proposition \ref{exist_app}.)

Here is a sketch of the proof of the main results.
First, we construct reflective modular forms on the Hermitian symmetric domain $D_L$ associated with a Hermitian lattice $L$ of signatures $(1,3)$, $(1,4)$, and $(1,5)$ by embedding $D_L$ into the Hermitian symmetric domain associated with the quadratic lattice $L_Q$ of signatures $(2,6)$, $(2,8)$, and $(2,10)$, and pulling back the strongly reflective modular forms constructed by Yoshikawa \cite{YoshikawaII}, and Gritsenko and Hulek  \cite{GH}.
Then, we apply the method of Gritsenko and Hulek \cite{Uniruledness} to ball quotients and show that some of them are uniruled.
Note that Gritsenko and Hulek used the numerical criterion of uniruledness due to Miyaoka and Mori \cite{MM}.

The outline of this paper is as follows.
In Section \ref{section:unitary Shimura varieties}, we introduce the notation about ball quotients and explain the embedding of $D_L$ into the Hermitian symmetric domain associated with $L_Q$.
In Section \ref{section:criterion}, we prove a criterion for the uniruledness of ball quotients in terms of reflective modular forms.
In Section \ref{Section:Kodaira}, we prove the main theorem of this paper.
In Appendix \ref{Section:Reflective modular forms}, we review some properties of reflective modular forms and recall the results of Yoshikawa, and Gritsenko and Hulek.
In Appendix \ref{example}, we give some examples of Hermitian lattices satisfying the assumptions of the main theorem.

\begin{rem}
We will only prove the uniruledness for ball quotients associated with Hermitian forms of signatures $(1,3)$, $(1,4)$, and $(1,5)$.
Let $L$ be the Hermitian lattice over $\OO_F$ of signature $(1,n)$.
We embed $D_L$ into the Hermitian symmetric domain associated with the quadratic lattice $L_Q$ of signature $(2,2n)$.
In fact, we can prove the uniruledness criterion for $n>1$ by using the method of Gritsenko and Hulek \cite{Uniruledness}; see Theorem \ref{uniruledness}.
However, in order to apply Theorem \ref{uniruledness}, we need reflective modular forms of high weight.
We shall construct reflective modular forms  by restricting the strongly reflective modular forms constructed by Yoshikawa, and  Gritsenko and Hulek. 
Yoshikawa constructed reflective modular forms on the Hermitian symmetric domain associated with quadratic forms of signature $(2,m)$ for $m<12$.
When $m=8$, the reflective modular forms constructed in \cite{YoshikawaII} are not strongly reflective except the  $L_Q=\mathbb{U}\oplus\mathbb{U}\oplus\mathbb{D}_6$ case.
Hence we can use strongly reflective modular forms constructed by Yoshikawa only when $m=8,10$, i.e., $n=4,5$.
 Gritsenko and Hulek constructed strongly reflective modular forms on the Hermitian symmetric domains associated with quadratic forms of signature $(1,3)$, $(1,4)$, and $(1,5)$, respectively.
Hence, we can prove the  uniruledness of ball quotients for $n=3,4,5$.
\end{rem}
\begin{rem}

\begin{enumerate}

\item  Kond\={o} \cite{Kondo1, Kondo2}, Gritsenko, Hulek and Sankaran \cite{GHS}, and Ma \cite{Ma} studied the Kodaira dimension of orthogonal modular varieties related to  moduli spaces of polarized K3 surfaces, and proved that some of them have non-negative Kodaira dimension or,  more precisely, orthogonal modular varieties are of general type if the polarization degree is sufficiently large.
  Kond\={o}, and Gritsenko, Hulek and Sankaran used modular forms of low weight vanishing on ramification divisors constructed by the Borcherds lift of the inverse of the Ramanujan delta function.
  Ma, Ohashi and Taki \cite{MOT} showed some ball quotients associated with Eisenstein lattices are rational.
  We study Hermitian lattices differently from them.

\item Recently, the author and Odaka \cite{MO} showed some orthogonal or unitary modular varieties are Fano  by using special reflective modular forms.
\end{enumerate}
\end{rem}

\subsection*{Acknowledgements}
The author would like to express his gratitude to his adviser, Tetsushi Ito, for his helpful comments and warm encouragement.
In particular, he told the author a lot of candidates of Hermitian lattices satisfying the conditions stated in the theorems.
The author would also like to thank Ken-Ichi Yoshikawa, Yuji Odaka and Haowu Wang for insightful suggestions and constructive discussions.
Yoshikawa explained the reflective modular forms he constructed in \cite{YoshikawaII} to the author.
Odaka referred the author to the paper \cite{HM} about rationally chain connected-ness.
Wang pointed out mistakes in an earlier
version of this paper.
Finally, he is deeply grateful to the
anonymous referees for helpful suggestions to make the exposition more readable.
This work is supported by JST ACT-X JPMJAX200P.

\section{Ball quotients}
\label{section:unitary Shimura varieties}
We start by studying the ramification of the canonical quotient map $\pi_{\Gamma}\colon D_L\to\F_L(\Gamma)$.
For a primitive element $r\in L$ with $\l r,r\r<0$, we define the  {\em reflection}  $\tau_r\in\mathbf{U}(L)(\Q)$ with respect to $r$ as follows:
\[\tau_r(\ell)\defeq\ell -\frac{2\l \ell,r\r}{\l r,r\r}r.\]

\begin{prop}
\label{ramification_divisors}
If $n>1$, then the ramification divisors of $\pi_{\Gamma}\colon D_L\to\F_L(\Gamma)$ contain the set
  \[\bigcup_{\substack{r\in L/\{\pm 1\},\ r\ \mathrm{is\ primitive}\\ \tau_r\in\Gamma \ \mathrm{or}\  -\tau_r\in\Gamma}}\{w\in D_L\mid \l w,r\r=0\}.\]
  If $F\neq\Q(\sqrt{-1})$ and $\Q(\sqrt{-3})$, then the ramification divisors are equal to the above set set-theoretically.
\end{prop}
\begin{proof}
The ramification divisors are fixed divisors by the action of quasi-reflections in $\Gamma$.
If $F\neq\Q(\sqrt{-1})$ and $\Q(\sqrt{-3})$,  then, by \cite[Corollary 3]{Behrens}, every quasi-reflection has order 2, so it is written as $\pm\tau_r$ for a primitive element $r\in L$; see also \cite[Corollary 2.13]{GHS}.
If $F=\Q(\sqrt{-1})$ (resp. $\Q(\sqrt{-3})$), quasi-reflections may have order $4$ (resp. $3$ or $6$).
\end{proof}
  Similarly, we define the reflection $\sigma_r\in\O^{+}(L_Q)$ for a primitive element $r\in L_Q$ as follows.
  \[\sigma_r(\ell)\defeq\ell-\frac{2(\ell,r)}{(r,r)}r.\]
  See the notation of orthogonal groups; see Appendix \ref{Section:Reflective modular forms}.

In the following, we recall how to embed the Hermitian symmetric domain associated with a Hermitian form into  the Hermitian symmetric domain associated with a quadratic form.
We follow Hofmann's paper \cite{Hofmann}.

Let $(L_Q,(\ ,\ ))$ be the quadratic lattice of rank $2n+2$ over $\Z$ of signature $(2,2n)$, where $L_Q$ is $L$ considered to be a free $\Z$-module and $(\ ,\ )\defeq\Tr_{F/\Q}\l\ ,\ \r$.
We put
\[\D_{L_Q}\defeq\left\{[w]\in\P(L_Q\otimes_{\Z}\C)\mid (w,w)=0,(w,\overline{w})>0\right\}^+.\]
Here, $+$ denotes one of its two connected components.
$\D_{L_Q}$ is the Hermitian symmetric domain associated with $\O^+(2,2n)$.
Then, we get an inclusion
\[\U(L)\hookrightarrow\O^+(L_Q)\]
so that
\begin{eqnarray}
  \label{injection}
  \iota\colon D_L\hookrightarrow \D_{L_Q}.
\end{eqnarray}
For $\lambda\in L$ with $\lambda\neq 0$, we define the {\em Heegner divisors} $H(\lambda)\subset D_L$ and $\H(\lambda)\subset\D_{L_Q}$ by
\begin{align*}
H(\lambda)&\defeq\{w\in D_L\mid \l \lambda,w\r=0\},\\
\H(\lambda)&\defeq\{w\in \D_{L_Q}\mid (\lambda,w)=0\}.
\end{align*}
This map (\ref{injection}) sends Heegner divisors on $D_L$ to Heegner divisors on $\D_{L_Q}$.
\begin{lem}
  \label{special}
  For $\lambda\in L$ with $\lambda\not=0$, we have 
  \[\iota(H(\lambda))=\iota(D_L)\cap \H(\lambda)\subset \D_{L_Q}\]
  as sets.
\end{lem}
\begin{proof}
  We follow \cite[Lemma 3]{Hofmann}.
  By \cite[section 4]{Hofmann}, to simplify the discussion, we may take $\iota(z)=z-\sqrt{-1}(iz)$ for $z\in L$, where $i$ denotes the endomorphism of $L_Q\otimes_{\Z}\R$ induced by multiplication by $\sqrt{-1}$ on the complex vector space $L\otimes_{\OO_F}\C$.
  Then, 
  \begin{align*}
  (\lambda,\iota(z))&=(\lambda,z-\sqrt{-1}(iz))\\
  &=(\lambda,z)-(\lambda,\sqrt{-1}(iz))\\
  &=(\lambda,z)-\sqrt{-1}(\lambda,(iz))\\
  &=2(\Re\l\lambda,z\r-\sqrt{-1}\Re\l\lambda,(iz)\r)\\
   &=2(\Re\l\lambda,z\r+\sqrt{-1}\Im\l\lambda,z\r).
  \end{align*}
 Hence, $ (\lambda,\iota(z))=0$ occurs if and only if $\l\lambda,z\r=0$ holds.
\end{proof}
\begin{lem}
  \label{special2}
  \begin{enumerate}
      \item  Let $\lambda,\lambda'\in L$ be elements satisfying 
 $\l \lambda,\lambda\r=\l \lambda',\lambda'\r\not=0$.
  Then,
  \[H(\lambda)=H(\lambda')\]
  if and only if $\lambda=a\lambda'$ for some $a\in\OO_F^{\times}$.
  \item  Let $\lambda,\lambda'\in L_Q$ be elements satisfying 
 $(\lambda,\lambda)=(\lambda',\lambda')$.
  Then, 
  \[\H(\lambda)=\H(\lambda')\]
  if and only if $\lambda=\pm\lambda'$.
  \end{enumerate}
\end{lem}
\begin{proof}
 (1) If $\lambda=a\lambda'$ for some $a\in\OO_F^{\times}$, then we have $\l v,\lambda\r=\overline{a}\l v,\lambda'\r$ for any $v\in L\otimes_{\OO_F}F$.
  Hence we have $H(\lambda)=H(\lambda')$.
  Conversely, if $H(\lambda)=H(\lambda')$, then $\lambda^{\perp}=(\lambda')^{\perp}$ in $L\otimes_{\OO_F}F$.
  Hence, $\lambda=a\lambda'$ for some $a\in F^{\times}$.
 Since $\l \lambda,\lambda\r=a\overline{a}\l \lambda',\lambda'\r\not=0$, we have $a\overline{a}=1$ and hence we have $a\in\OO_F^{\times}$.
 
 Claim (2) can be shown in a similar way.
\end{proof}

Moreover, the map (\ref{injection}) sends cusps of $D_L$ with respect to $\U(L)$ to cusps of $\D_{L_Q}$ with respect to $\O^+(L_Q)$, so the pullback of a cusp form on $\D_{L_Q}$ to $D_L$ by (\ref{injection}) is also a cusp form; 
see \cite[Proposition 2]{Hofmann} for details.

\section{Uniruledness criterion}
\label{section:criterion}
Let $M_k(\Gamma,\chi)$ be the space of modular forms of weight $k$ on $D_L$ with respect to a finite index subgroup $\Gamma\subset\U(L)$ and a character $\chi\colon\U(L)\to\C^{\times}$, similarly to the orthogonal case; see Appendix \ref{Section:Reflective modular forms}.
We also define the reflective modular forms on $D_L$ in the same way as the orthogonal case.
We can prove that certain ball quotients are uniruled under some assumption on the vanishing orders and weights of reflective modular forms.
\begin{thm}
\label{uniruledness}
Let $F$ be an imaginary quadratic field and $(L,\l\ ,\ \r)$ be a Hermitian lattice over $\OO_F$ of signature $(1,n)$ for $n>1$.
Let $a,k > 0$ be positive integers satisfying $k > a(n+1)$.
If there exists a non-zero reflective modular form $F_{a,k}\in M_k(\Gamma,\chi)$ of weight $k$ for which the multiplicity of every irreducible component of $\div(F_{a,k})$ is less than or equal to $a$, then the ball quotient $\F_L(\Gamma')$ is uniruled for every arithmetic group $\Gamma'$ containing $\Gamma$.
\end{thm}
\begin{proof}
It suffices to show the theorem for $\Gamma=\Gamma'$.
  First, we assume $F\neq\Q(\sqrt{-1})$ and $\Q(\sqrt{-3})$.
  Then by \cite[Corollary 3]{Behrens}, the ramification divisors are fixed divisors defined by reflections.
  We take the canonical toroidal compactification $X\defeq\overline{\F_L(\Gamma)}$ of $\F_L(\Gamma)$.
  We can prove the theorem in the same way as \cite[Theorem 2.1]{Uniruledness} by using the numerical criterion of uniruledness due to Miyaoka and Mori \cite{MM}. 
  We denote the ramification divisors in $X$ by $B=\sum_r B_r$.
  Let $D=\sum_{\alpha} D_{\alpha}$ be the boundary and $L$ be the line bundle of modular forms of weight 1.
  Then, we have
  \begin{eqnarray}
  \label{bundle_1}
    K_{X}=(n+1)L-\frac{1}{2}B-D.
  \end{eqnarray}
  
  By the assumption on $F_{a,k}$, we have
  \begin{align}
  \label{bundle_2}
    kL&=\frac{1}{2}\sum_r m_rB_r+\sum_{\alpha}\delta_{\alpha}D_{\alpha}\notag\\
    &=\frac{1}{2}\left\lbrace mB+\sum_r (m_r-m)B_r\right\rbrace+\sum_{\alpha}\delta_{\alpha}D_{\alpha}
  \end{align}
  for some non-negative rational numbers $m_r$ and  $\delta_{\alpha}$ with $m_r\leq a$.
  By substituting (\ref{bundle_2}) for (\ref{bundle_1}), we have
  \begin{eqnarray}
  \label{-K}
  -K_{X}=\left\lbrace\frac{k}{m}-n-1\right\rbrace L+\sum_r\frac{m-m_r}{2m}B_r+\sum_{\alpha}\frac{m-\delta_{\alpha}}{m}D_{\alpha}.
  \end{eqnarray} 
  For a resolution $\widetilde{X}\to X$, we obtain 
   \[-K_{\widetilde{X}}=\left\lbrace\frac{k}{m}-n-1\right\rbrace L+\sum_r\frac{m-m_r}{2m}B_r+\sum_{\alpha}\frac{m-\delta_{\alpha}}{m}D_{\alpha}+\sum_{\beta}\epsilon_{\beta}E_{\beta}\]
   where $E_{\beta}$ are the exceptional divisors and $\epsilon_{\beta}$ are non-negative integers.
   Here, $B_r$ and $D_{\alpha}$ denote the strict transforms of the corresponding divisors on $X$.
   In the same way as the proof of \cite[Theorem 2.1]{Uniruledness}, for a curve $C$ on $\widetilde{X}$ which does not pass through the boundary and the singular locus of $\F_L(\Gamma)$, it suffices to show  $K_{\widetilde{X}}\cdot C<0$.
   Now 
    \[-K_{\widetilde{X}}\cdot C=\left\lbrace\left(\frac{k}{m}-n-1\right)L+\sum_r\frac{m-m_r}{2m}B_r+\sum_{\alpha}\frac{m-\delta_{\alpha}}{m}D_{\alpha}+\sum_{\beta}\epsilon_{\beta}E_{\beta}\right\rbrace\cdot C.\]
    From the assumption on $C$, we have $D_{\alpha}\cdot C=E_{\beta}\cdot C=0$.
    We also have $\frac{m-m_r}{2m}B_r\cdot C>0$ and $\left(\frac{k}{m}-n-1\right)L\cdot C>0$ by \cite[Theorem 2.1]{Uniruledness}, so it follows that $K_{\widetilde{X}}\cdot C<0$.
    This implies $\F_L(\Gamma)$ is uniruled by the numerical criterion of Miyaoka and Mori.
    When $F=\Q(\sqrt{-1})$ or $\Q(\sqrt{-3})$, the ramification divisors might be defined by the elements of order 4, or 3 and 6.
    The same calculation as above holds in these cases, from which we can conclude that $\F_L(\Gamma)$ is uniruled.

\end{proof}

\begin{rem}
 Gritsenko, Hulek and Sankaran showed that the ramification divisors of orthogonal modular varieties of dimension larger than 2 are induced by $g\in\Gamma$ such that $g$ or $-g$ is a reflection; see \cite[Corollary 2.13]{GHS}.
 Hence, the uniruledness criterion \cite[Theorem 2.1]{Uniruledness} for orthogonal modular varieties is proved for $n>2$.
 On the other hand, Behrens \cite[Corollary 3]{Behrens} proved the same claim for ball quotients for $n>1$, so Theorem \ref{uniruledness} needs this restriction.
\end{rem}

As we stated in Introduction, we get the stronger result of rationally chain connected-ness.
Before proving them, let us define the notion of naked cusps.
We call a cusp \textit{naked} if it not contained in all branch divisors; see \cite[Definition 2.7]{MO}.
Now, we shall show that the Satake-Baily-Borel compactification of $\F_L(\Gamma)$ is rationally chain connected modulo naked cusps according to \cite{HM}.
For the precise definition, see \cite[Section 1, 2]{HM}.

\begin{thm}
    \label{rem:rationally_chain}
    In addition to the assumption in Theorem \ref{uniruledness}, suppose that $F_{a,k}$ is a special reflective modular form in the sense of \cite[Assumption 2.1 (1)]{MO}.
    Then, the Satake-Baily-Borel compactification of $\F_L(\Gamma)$ is rationally chain connected modulo naked cusps.
\end{thm}
\begin{proof}
Since $F_{a,k}$ is a special reflective modular form, $\div(F_{a,k})$ is a constant multiple of $B$ in the $\Q$-Picard group.
Hence, by the same method as in the proof of \cite[Theorem 2.4]{MO}, there exists a positive rational number $s\in\Q$ (denoted by $s(x)-1$ in \cite{MO}) so that
\[-K_{\overline{\F_L(\Gamma)}^{\mathrm{SBB}}}=sL\]
on the Satake-Baily-Borel compactification $\overline{\F_L(\Gamma)}^{\mathrm{SBB}}$.
Here, we used the result concerning Hilzebruch's proportionality principle by Mumford and the construction of the Stake-Baily-Borel compactification. 
Note that the log pair $(\overline{\F_L(\Gamma)}^{\mathrm{SBB}}, 0)$ is log canonical and the locus of log canonical singularities are naked cusps; see \cite[Theorem 2.4, Lemma 2.9]{MO}.
Therefore, it follows that $\overline{\F_L(\Gamma)}^{\mathrm{SBB}}$ is rationally chain connected modulo naked cusps from \cite[Theorem 1.2]{HM}.
\end{proof}
\begin{rem}
\label{rem:rationally_kondo}
Note that not all varieties appearing in this paper are not known to be rationally chain connected. When the branch divisors are known, such as in the case of \cite{Kondo3}, we can conclude that they are rationally chain connected.
    
\end{rem}

\section{Applications to the uniruledness of ball quotients}
\label{Section:Kodaira}
In this section, we show certain ball quotients associated with Hermitian forms of signatures $(1,3)$, $(1,4)$, and $(1,5)$ are uniruled.
Note that for $F\neq\Q(\sqrt{-1})$ and  $\Q(\sqrt{-3})$, we have $\OO_F^{\times}/\{\pm 1\} =\{1\}$.
On the other hand, we have $\OO_{\Q(\sqrt{-1})}^{\times}/\{\pm 1\} =\{1,\sqrt{-1}\}$ and $\OO_{\Q(\sqrt{-3})}^{\times}/\{\pm 1\} =\{1,\omega,\omega^2\}$ where $\omega$ is a primitive third root of unity.
For modular forms on $\D_{L_Q}$, see Appendix \ref{Section:Reflective modular forms}.

\begin{thm}
  \label{Kodaira_dim_app}
Assume that
  \begin{enumerate}
  \item $n=5$.
      \item $L_Q$ is even $2$-elementary, $\delta(L_Q)=0$
  and $\ell(L_Q)\leq 8$.
  Moreover, $\ell(L_Q)\leq 6$ if $F=\Q(\sqrt{-3})$.
      \item  $2\cdot\Div(r)\subset\OO_F$ for any primitive $r\in L$ with $\l r,r\r=-1$.
  \end{enumerate}
  Then $X_L$ is uniruled.
\end{thm}
\begin{proof}
 Since $L_Q$ is $2$-elementary with $\ell(L_Q)\leq 8$ and $\delta(L_Q)=0$, we have a strongly reflective modular form $\Psi_{L_Q}$ on $\D_{L_Q}$ for $\O^+(L_Q)$ by Theorem \ref{Yoshikawa}.
 This has zeros on Heegner divisors defined by $(-2)$-vectors only, i.e.,
 \[\div(\Psi_{L_Q})=\sum_{\Delta_{L_Q}/\{\pm 1\}}\H(r).\]
Let $\Xi\defeq\{r\in L\mid\ \l r,r\r=-1\}.$

First, we consider the case of  $F\neq\Q(\sqrt{-1})$ and $\Q(\sqrt{-3})$.
We have
\[\div(\iota^{\star}\Psi_{L_Q})=\sum_{r\in\Xi/\{\pm 1\}}H(r).\]
From assumption, for any primitive $r\in L$ with $\l r,r\r=-1$, we have $2\l\ell,r\r\in\OO_F$, where $\ell$ is  any element of $L$.
Hence, for such $r$, we have 
\[\frac{2\l\ell,r\r}{\l r,r\r}\in\OO_F,\] 
so the Heegner divisors $H(r)$ appearing in $\div(\iota^{\star}\Psi_{L_Q})$ is contained in the ramification divisors of $\pi_{\U(L)}\colon D
_L\to X_L$.
Therefore $\iota^{\star}\Psi_{L_Q}$ is a reflective modular form of weight $w(\iota^{\star}\Psi_{L_Q})\geq 12$ on $D_L$, and its vanishing order on ramification divisors is at most 1.

Now, let $F=\Q(\sqrt{-1})$ or $\Q(\sqrt{-3})$.
In this case, from a similar discussion to the one of  $F\neq\Q(\sqrt{-1}), \Q(\sqrt{-3})$, it follows that $\iota^{\star}\Psi_{L_Q}$ is a reflective modular form.
Hence, we need to consider its weight and vanishing orders.
By Lemma \ref{special2} (2), we have
\[\div(\iota^{\star}\Psi_{L_Q})=\sum_{r\in\Xi/\{\pm 1\}}H(r)=
\begin{cases}
\displaystyle{\sum_{\Xi/\OO^{\times}_{\Q(\sqrt{-1})}}}2H(r)&(F=\Q(\sqrt{-1})),\\
\displaystyle{\sum_{\Xi/\OO^{\times}_{\Q(\sqrt{-3})}}}3H(r)&(F=\Q(\sqrt{-3})).
\end{cases}
\]
On the other hand, in each case, the assumption leads to the weight $w(\iota^{\star}\Psi_{L_Q})\geq 28$.

Consequently, we can easily check that $\iota^{\star}\Psi_{L_Q}$ and $L$ satisfy the condition in Theorem \ref{uniruledness}, so conclude that $X_L$ is uniruled.
\end{proof}
\begin{rem}
\begin{enumerate}
    \item The only case listed in Table \ref{table:lattices_and_invariants} to which Theorem \ref{Kodaira_dim_app} does not apply is $M=\mathbb{U}\oplus \mathbb{U}(2)\oplus \mathbb{E}_8(-2)$.
    \item In Appendix \ref{example}, we give examples of Hermitian lattices  $(L,\l\ ,\ \r)$ satisfying the conditions of Theorem \ref{Kodaira_dim_app}.
  There, we explain what happens in the $\delta=1$ case.
  For an even $2$-elementary quadratic lattice $M$ of signature $(2,10)$ with $\delta(M)=1$, Yoshikawa constructed a reflective modular form $\Psi_M$ of weight $4g$, where $g\defeq 2^{(12-\ell(M))/2}+1$; see \cite[Theorem 8.1]{YoshikawaII}.
  The multiplicity of the divisors of this function is at most $g$.
  In particular, $\Psi_M$ is not a  strongly reflective modular form.
  Moreover, $\Psi_M$ does not  satisfy the assumption in Theorem \ref{uniruledness}, because $4g/g=4$ is not greater than $5$.
  Hence, we cannot use $\Psi_M$ to study the uniruledness of  ball quotients.
  \end{enumerate}
\end{rem}

\begin{thm}
\label{Kodaira_dim_app2}
    Assume that
  \begin{enumerate}
      \item $L_Q\cong\mathbb{U}\oplus\mathbb{U}(2)\oplus\mathbb{E}_8(-2)$.
      \item  $\Div(r)\subset\OO_F$ for any primitive $r\in L$ with $\l r,r\r=-2$ and $\Tr_{F/\Q}(\Div(r))\subset 2\Z$.
  \end{enumerate}
  Then  $X_L$ is uniruled.
\end{thm}
\begin{proof}
By Theorem \ref{enriques_reflective}, we have a strongly reflective $\Phi_{124}$ of weight 124 on $\D_N$ satisfying 
\[\div(\Phi_{124})=\sum_{\substack{r\in N/\{\pm 1\}\ \mathrm{s.t.}\ (r,r)=-4 \\ 2|(\ell,r)\ \mathrm{for\ any\ }\ell\in N}}\H(r).\]
Similar to the proof of Theorem \ref{Kodaira_dim_app}, the modular form $\iota^{\star}\Phi_{124}$ on $D_L$ can be shown to be reflective and weight $w(\iota^{\star}\Phi_{124})=124$.
For each $F$, the vanishing order of $\iota^{\star}(\Phi_{124})$ along the ramification divisors is 3, so by  Theorem \ref{uniruledness}, $X_L$ is uniruled. 
\end{proof}

\begin{thm}
\label{Kodaira_dim_app3}
    Assume that
  \begin{enumerate}
      \item $L_Q\cong\mathbb{U}\oplus\mathbb{U}\oplus\mathbb{D}_6$.
      \item  $2\cdot\Div(r)\subset\OO_F$ for any primitive $r\in L$ with $\l r,r\r=-1$.
  \end{enumerate}
  Then $X_L$ is uniruled.
\end{thm}
\begin{proof}
By Theorem \ref{Yoshikawa2}, we have a strongly reflective modular form $\Psi_{L_Q}$ of weight 102 on $\D_{L_Q}$.
 This has zeros on Heegner divisors defined by $(-2)$-vectors only, i.e.,
 \[\div(\Psi_{L_Q})=\sum_{\Delta_{L_Q}/\{\pm 1\}}\H(r).\]
Similar to the proof of Theorem \ref{Kodaira_dim_app}, the modular form $\iota^{\star}\Psi_{L_Q}$ on $D_L$ can be shown to be reflective and weight $w(\iota^{\star}\Psi_{L_Q})=102$.
For each $F$, the vanishing order of $\iota^{\star}\Phi_{124}$ along the ramification divisors is 3, so by  Theorem \ref{uniruledness}, $X_L$ is uniruled. 
\end{proof}
\begin{thm}
\label{Kodaira_dim_app4}
  Assume that
  \begin{enumerate}
      \item $L_Q\cong\mathbb{U}\oplus\mathbb{U}\oplus\mathbb{D}_4$, $\mathbb{U}\oplus\mathbb{U}(2)\oplus\mathbb{D}_4$ or $\mathbb{U}(2)\oplus\mathbb{U}(2)\oplus\mathbb{D}_4$.
      \item  $2\cdot\Div(r)\subset\OO_F$ for any primitive $r\in L$ with $\l r,r\r=-1$.
  \end{enumerate}
  Then $X_L$ is uniruled.
\end{thm}
\begin{proof}
This can be proved in the same way as Theorem \ref{Kodaira_dim_app3} by using Theorem   \ref{Yoshikawa2}.
\end{proof}

\appendix

\section{Reflective modular forms on $\D_{L_Q}$}
\label{Section:Reflective modular forms}
Here, let us introduce reflective modular forms and recall some examples.
Note that the theorems appearing in this section are nothing new; they were proved by Gritsenko and  Hulek \cite{GH}, and  Yoshikawa \cite{YoshikawaII}. 

For a quadratic lattice $M$ over $\Z$ of signature $(2,m)$, let $\mathbf{O}(M)$ be the orthogonal groups scheme over $\Z$ asociated with $M$.
We denote by $\O^+(M)$ the index 2 subgroup of the integral orthogonal group $\mathbf{O}(M)(\Z)$ preserving $\D_M$.
Let $\M_k(\Gamma,\chi)$ and $\S_k(\Gamma,\chi)$ be the spaces of modular forms and cusp forms of weight $k$ on $\D_M$ with respect to $\Gamma$ and a character $\chi:\Gamma\to\C^{\times}$, where  $\Gamma\subset\O^+(M)$ is a finite index subgroup.

In the following, we assume $m>2$.
Then, in the same way as Proposition \ref{ramification_divisors}, by \cite[Corollary 2.13]{GHS}, the ramification divisors of  \[\pi_{\Gamma}\colon\D_{M}\to\Gamma\backslash\D_{M}\]
are equal to 
\[\bigcup_{\substack{r\in M/\{\pm 1\},\ r\ \mathrm{is\ primitive}\\ \sigma_r\in\Gamma \ \mathrm{or}\  -\sigma_r\in\Gamma}}\H(r).\]
Recall that a modular form $F_k\in\M_k(\Gamma,\chi)$ is called {\em reflective} if
  \[\Supp(\div F_k)\subset \bigcup_{\substack{r\in M/\{\pm 1\},\ r\ \mathrm{is\ primitive}\\ \sigma_r\in\Gamma \ \mathrm{or}\  -\sigma_r\in\Gamma}}\H(r),\]
  set-theoretically.
  We call a reflective modular form $F_k$ {\em strongly reflective} if the multiplicity of any irreducible component of $\div(F_k)$ is 1.
  Reflective modular forms were introduced by Gritsenko and Nikulin \cite{GN} and studied in \cite{Reflective} for details.

Now, let us give some examples of reflective modular forms that will be used in the proof of the main theorem of this paper.
First, we introduce some invariants of quadratic lattices according to \cite{YoshikawaII}.
Let $r(M)$ be the rank of $M$, and $M^{\vee}$ be its dual lattice, i.e., \[M^{\vee}\defeq\{v\in M\otimes_{\Z}\Q\mid (v,w)\in\Z\ \mathrm{for\ any\ }w\in M\}.\]
We set
\[\Delta_M\defeq\{v\in M\mid (v,v)=-2\}\]
and define
\[\Delta_M'\defeq\left\{v\in \Delta_M\mid \frac{v}{2}\not\in M^{\vee}\right\},\]
\[\Delta_M''\defeq\left\{v\in \Delta_M\mid \frac{v}{2}\in M^{\vee}\right\},\]
so that $\Delta_M=\Delta_M'\sqcup\Delta_M''$.
We call $M$ {\em even} if $(v,v)\in 2\Z$ for any $v\in M$.
We say $M$ is {\em $2$-elementary} if there exists a non-negative integer $\ell(M)$ such that
\[M^{\vee}/M\cong (\Z/2\Z)^{\ell(M)}.\]
Then, $r(M)\geq\ell(M)$ and $r(M)\equiv\ell(M)\bmod 2$.
We also define
\[\delta(M)\defeq\begin{cases}0&((v,v)\in\Z\ \mathrm{for\ any}\ v\in M^{\vee})\\
    1&((v,v)\not\in\Z\ \mathrm{for\ some}\ v\in M^{\vee}).
    \end{cases}\]
By \cite[Theorem 3.6.2]{Nikulin},  the triplet $(\mathrm{sign}(M),\ell(M),\delta(M))$ determines the isometry class of an indefinite even $2$-elementary lattice $M$.

In the following, we assume that the signature of $M$ is $(2,10)$ or $(2,8)$, and $M$ is an even $2$-elementary lattice.
In \cite{YoshikawaII}, Yoshikawa  constructed a $\C[M^{\vee}/M]$-valued holomorphic function $F_M$ on the complex upper half plane and calculated its Borcherds lift $\Psi_M$.
We review Yoshikawa's result for a special case.
\begin{thm}[Yoshikawa]
  \label{Yoshikawa}
  Let $M$ be an even $2$-elementary  quadratic lattice over $\Z$ of signature $(2,10)$ and  $\delta(M)=0$.
  Then, $\Psi_M$ is a modular form on $\D_M$ for $\O^+(M)$ satisfying
  \[\div(\Psi_M)=\sum_{\lambda\in\Delta'_M/\{\pm 1\}}\H(\lambda).\]
  Moreover, the weight $w(M)$ of $\Psi_M$ is given by
  \[w(M)=4(2^{(12-\ell(M))/2}+1)-8=2^{(16-\ell(M))/2}-4.\]
  In particular, since for any primitive $(-2)$-vector $r\in M$, the reflection $\sigma_r$ is always an element of  $\O^{+}(M)$, the reflective modular form $\Psi_M$ is strongly reflective.
\end{thm}
\begin{proof}
  See \cite[Theorem 8.1]{YoshikawaII}.
\end{proof}
\begin{thm}[Yoshikawa]
  \label{Yoshikawa2}
  Let $M\defeq\mathbb{U}\oplus\mathbb{U}\oplus\mathbb{D}_6$ be a quadratic lattice over $\Z$ of signature $(2,8)$.
  Then, $\Psi_M$ is a modular form on $\D_M$ for $\O^+(M)$ satisfying
  \[\div(\Psi_M)=\sum_{\lambda\in\Delta'_M/\{\pm 1\}}\H(\lambda).\]
  Moreover, the weight $w(M)$ of $\Psi_M$ is 102 and $\Psi_M$ is strongly reflective.
\end{thm}
\begin{proof}
  See \cite[Theorem 8.1]{YoshikawaII}.
\end{proof}

\begin{rem}
\begin{enumerate}
    \item Let us list $(M,( \ ,\ ))$ satisfying the condition of Theorem \ref{Yoshikawa}.
Under the assumption of the existence of a primitive embedding into the K3 lattice, these quadratic lattices are classified in \cite{YoshikawaIV}.
Table \ref{table:Lattices_and_weights_of_modular_forms} lists the even $2$-elementary  quadratic lattices of signature $(2,10)$ satisfying  $\delta(M)=0$  which have primitive embeddings into the K3 lattice.
\vspace{0.2in}
\begin{center}
\begin{table}[h]
\caption{Lattices and weights of modular forms}
\label{table:Lattices_and_weights_of_modular_forms}
\renewcommand{\arraystretch}{1.5}
    \begin{tabular}{|c||c|c|} \hline
     Quadratic lattices $M$&$\ell(M)$&$w(M)$\\ \hline
      $\mathbb{U}\oplus \mathbb{U}(2)\oplus \mathbb{E}_8(-2)$&$10$&4\\ \hline
      $\mathbb{U}\oplus \mathbb{U}\oplus \mathbb{E}_8(-2)$&$8$&12\\\hline
      $\mathbb{U}\oplus \mathbb{U}(2)\oplus \mathbb{D}_4(-1)\oplus \mathbb{D}_4(-1)$&$6$&28\\\hline
      $\mathbb{U}\oplus \mathbb{U}\oplus \mathbb{D}_4(-1)\oplus \mathbb{D}_4(-1)$&$4$&60\\\hline
      $\mathbb{U}\oplus \mathbb{U}\oplus \mathbb{D}_8(-1)$&$2$&124\\\hline
      $\mathbb{U}\oplus \mathbb{U}\oplus \mathbb{E}_8(-1)$&$0$&252\\\hline
    \end{tabular}
  \renewcommand{\arraystretch}{1.0}
  \end{table}
  \end{center}
\vspace{0.2in}

\item  Even unimodular lattices over $\Z$ of signature $(2,10)$ are $2$-elementary and $\delta=0$, hence they satisfy the assumption in Theorem \ref{Yoshikawa}.
\end{enumerate}
\end{rem}

There exists a strongly reflective modular form on $\D_M$ different from Yoshikawa's one, where $M\defeq\mathbb{U}\oplus \mathbb{U}(2)\oplus \mathbb{E}_8(-2)$.
Let $(M\defeq \mathbb{U}\oplus \mathbb{U}(2)\oplus \mathbb{E}_8(-2),(\ ,\ ))$ be an even quadratic lattice of signature $(2,10)$.
The quadratic lattice $M$ is related to the moduli space of polarized Enriques surfaces.
Gritsenko and Hulek \cite{GH} constructed a strongly reflective modular form on $\D_M$.
\begin{thm}[Gritsenko and Hulek]
  \label{enriques_reflective}
  There exists a strongly reflective modular form  $\Phi_{124}\in \S_{124}(\O^+(M),\mu_2)$ of weight $124$ where $\mu_2$ is a binary character of $\O^+(M)$.
  This satisfies
  \[\div(\Phi_{124})=\bigcup_{\substack{r\in M/\{\pm 1\}\ \mathrm{s.t.}\ (r,r)=-4 \\ 2|(\ell,r)\ \mathrm{for\ any\ }\ell\in M}}\H(r)\subset \bigcup_{\substack{r\in M/\{\pm 1\}\ \mathrm{s.t.}\ (r,r)=-4 \\ 2|(\ell,r)\ \mathrm{for\ any\ }\ell\in M}}\H(r)\ \bigsqcup\bigcup_{\substack{r\in M/\{\pm 1\}\\ \mathrm{s.t.}\ (r,r)=-2}}\H(r),\]
  where the right-hand side is the ramification divisor of $\pi\colon\D_M\to\O^+(M)\backslash\D_M$
\end{thm}
\begin{proof}
  See \cite[Lemma 5.4]{GH}.
\end{proof}
\begin{thm}[Yoshikawa]
  \label{Yoshikawa3}
  Let $M\defeq\mathbb{U}\oplus\mathbb{U}\oplus\mathbb{D}_4$, $\mathbb{U}\oplus\mathbb{U}(2)\oplus\mathbb{D}_4$ or $\mathbb{U}(2)\oplus\mathbb{U}(2)\oplus\mathbb{D}_4$ be a quadratic lattice over $\Z$ of signature $(2,6)$.
  Then, $\Psi_M$ is a strongly reflective modular form on $\D_M$ for $\O^+(M)$ satisfying
  \[\div(\Psi_M)=\sum_{\lambda\in\Delta'_M/\{\pm 1\}}\H(\lambda).\]
  Moreover, the weight $w(M)$ of $\Psi_M$ is as follows.
  \[w(M)=8(2^{\frac{8-\ell(M)}{2}}+1).\]
\end{thm}
\begin{proof}
  See \cite[Theorem 8.1]{YoshikawaII}.
\end{proof}
Table \ref{table:Lattices_and_weights_of_modular_forms_II} is a list of the triplets $(M,\ell(M),w(M))$ satisfying the condition of Theorem \ref{Yoshikawa}.
\vspace{0.2in}
\begin{center}
\renewcommand{\arraystretch}{1.5}
\begin{table}[h]
\caption{Lattices and weights of modular forms II}
\label{table:Lattices_and_weights_of_modular_forms_II}
    \begin{tabular}{|c||c|c|} \hline
     Quadratic lattices $M$&$\ell(M)$&$w(M)$\\ \hline
      $\mathbb{U}(2)\oplus \mathbb{U}(2)\oplus \mathbb{D}_4$&$6$&24\\ \hline
      $\mathbb{U}\oplus \mathbb{U}(2)\oplus \mathbb{D}_4$&$4$&40\\\hline
      $\mathbb{U}\oplus \mathbb{U}\oplus \mathbb{D}_4$&$2$&72\\\hline
    \end{tabular}
    \end{table}
  \renewcommand{\arraystretch}{1.0}
  \end{center}
\vspace{0.2in}

\section{Examples of Hermitian lattices}
\label{example}
In this appendix, we list some Hermitian lattices that satisfy the conditions of the theorems in Section \ref{Section:Kodaira}.
We start by defining some special Hermitian lattices.

Let
$L_{\mathbb{U}\oplus\mathbb{U}}$
be the Hermitian lattice defined by
\[\gamma
\begin{pmatrix}
0 & 1 \\
-1 & 0 \\
\end{pmatrix}\]
where
\[\gamma = \begin{cases}
  \frac{1}{2\sqrt{d}}&(d\equiv 2,3 \bmod 4) \\
  \frac{1}{\sqrt{d}}&(d\equiv 1 \bmod 4).
\end{cases}\]
Then, the associated quadratic lattice
$(L_{\mathbb{U}\oplus\mathbb{U}})_Q$
is isometric to $\mathbb{U}\oplus\mathbb{U}$.
(See \cite[Example 2.1]{Hofmann}.)

Let $L_{-1,\mathbb{U}\oplus\mathbb{U}(2)}$ be the Hermitian lattice over $\OO_{\Q(\sqrt{-1})}$ defined by
\[
\begin{pmatrix}
0 & \frac{1+\sqrt{-1}}{2} \\
\frac{1-\sqrt{-1}}{2} & 0
\end{pmatrix}
\]
Then, the associated quadratic lattice $(L_{-1,\mathbb{U}\oplus\mathbb{U}(2)})_Q$
is defined by
\[
\begin{pmatrix}
0 & 0 & 1 & 1 \\
0 & 0 & -1 & 1 \\
1 & -1 & 0 & 0 \\
1 & 1 & 0 & 0
\end{pmatrix}.
\]
It is isometric to $\mathbb{U}\oplus\mathbb{U}(2)$.

Let $L_{-2,\mathbb{U}\oplus\mathbb{U}(2)}$ be the Hermitian lattice over $\OO_{\Q(\sqrt{-2})}$ defined by
\[
\begin{pmatrix}
0 & \frac{1}{2} \\
\frac{1 }{2} & 0
\end{pmatrix}
\]
Then. the associated quadratic lattice $(L_{-2,\mathbb{U}\oplus\mathbb{U}(2)})_Q$
is defined by
\[
\begin{pmatrix}
0 & 0 & 1 & 0 \\
0 & 0 & 0 & 2 \\
1 & 0 & 0 & 0 \\
0 & 2 & 0 & 0
\end{pmatrix}.
\]
It is isometric to $\mathbb{U}\oplus\mathbb{U}(2)$.

Coulageon-Nebe \cite{CN} constructed Hermitian lattices  $L_{\mathbb{D}_4}$ of rank 4 such that $(L_{\mathbb{D}_4})_Q\cong\mathbb{D}_4$ for $F=\Q(\sqrt{-1})$ and  $\Q(\sqrt{-2})$.
Let $L_{-1,\mathbb{D}_4}$ be the Hermitian lattice over $\OO_{\Q(\sqrt{-1})}$  defined by  \[\begin{pmatrix}
1 & \frac{1}{1+\sqrt{-1}} \\
\frac{1}{1-\sqrt{-1}} & 1 \\
\end{pmatrix}.
\]
Then, the associated quadratic lattice $(L_{-1,\mathbb{D}_4})_Q$ is defined by 
\[\begin{pmatrix}
2 & 0 & 1 & -1 \\
0 & 2 & 1 & 1 \\
1 & 1 & 2 & 0 \\
-1 & 1 & 0 & 2 \\
\end{pmatrix}.\]
It is isometric to $\mathbb{D}_4$.

Let $L_{-2,\mathbb{D}_4}$ be the Hermitian lattice over $\OO_{\Q(\sqrt{-2})}$  defined by  \[\begin{pmatrix}
1 & \frac{1+\sqrt{-2}}{2} \\
\frac{1-\sqrt{-2}}{2} & 1 \\
\end{pmatrix}.
\]
Then the associated quadratic lattice $(L_{-2,\mathbb{D}_4})_Q$ is defined by 
\[\begin{pmatrix}
2 & 0 & 1 & 2 \\
0 & 4 & -2 & 2 \\
1 & -2 & 2 & 0 \\
2 & 2 & 0 & 4 \\
\end{pmatrix}.\]
It is isometric to $\mathbb{D}_4$.

Let $L_{\mathbb{D}_6}$ be the Hermitian lattice over $\OO_{\Q(\sqrt{-1})}$  defined by  \[\frac{1}{2}\begin{pmatrix}
2 & 1+\sqrt{-1} &1\\
1-\sqrt{-1} & 2 &0\\
1&0&2\\
\end{pmatrix}.
\]
Then, the associated quadratic lattice $(L_{\mathbb{D}_6})_Q$ is defined by 
\[\begin{pmatrix}
2 & 0 & 1 & 1&1&0 \\
0 & 2 & -1 & 1&0&1 \\
1 & -1 & 2 & 0&0&0 \\
1 & 1 & 0 & 2 &0&0 \\
1 &0  &0  &0 &2&0\\
0 &1 &0  &0 &0&2\\
\end{pmatrix}.\]
It is isometric to $\mathbb{D}_6$.

Let $L_{\mathbb{D}_8}$ be the Hermitian lattice over $\OO_{\Q(\sqrt{-1})}$  defined by  \[\frac{1}{2}\begin{pmatrix}
2 & 1+\sqrt{-1} &0&0\\
1-\sqrt{-1} & 2 &1&0\\
0&1&2&1\\
0&0&1&2\\
\end{pmatrix}.
\]
Then, the associated quadratic lattice $(L_{\mathbb{D}_4})_Q$ is defined by 
\[\begin{pmatrix}
2 & 0 & 1 & 1&0&0&0&0 \\
0 & 2 & -1 & 1&0&0&0&0 \\
1 & -1 & 2 & 0&1&0&0&0 \\
1 & 1 & 0 & 2&0&1&0&0 \\
0&0&1&0&2&0&1&0\\
0&0&0&1&0&2&0&1\\
0&0&0&0&1&0&2&0\\
0&0&0&0&0&1&0&2\\
\end{pmatrix}.\]
It is isometric to $\mathbb{D}_8$.

Here, we give Hermitian lattices $(L_{d, \mathbb{E}_8}, A_d)$ over $\OO_{\Q(\sqrt{d})}$ of rank $(4, 0)$ whose associated quadratic lattices $B_d\cong\mathbb{E}_8$ are even unimodular for $d=-1, -2$.
 The matrix $A_{-1}$ is called Iyanaga’s matrix \cite{Iyanaga} and $A_{-2}$ is taken from \cite{Hentschel}. 
 In \cite{Classification}, Hentschel, Krieg and Nebe classified the Hermitian lattices over $\OO_F$ for the remaining imaginary quadratic fields of class number 1.

\begin{align*}
    A_{-1}&=\frac{1}{2}\begin{pmatrix}
2 & -\sqrt{-1} & -\sqrt{-1}&1\\
\sqrt{-1} & 2 & 1 & \sqrt{-1}\\
\sqrt{-1} & 1 & 2 & 1\\
1& -\sqrt{-1} & 1 & 2\\
\end{pmatrix}\\
B_{-1}&=\begin{pmatrix}
2 & 0 & 0 & -1& 0 & -1& 1 & 0 \\
0 & 2 & 1 & 0 & 1 & 0 & 0 & 1 \\
0 & 1 & 2 & 0 & 1 & 0 & 0 & 1 \\
-1& 0 & 0 & 2 & 0 & 1 &-1& 0 \\
0 & 1 & 1 & 0 & 2 & 0 & 1 & 0 \\
-1& 0 & 0 & 1 & 0 & 2 & 0 & 1 \\
1 & 0 & 0 & -1& 1 & 0 & 2 & 0 \\
0 & 1 & 1 & 0 & 0 & 1 & 0 & 2 \\
\end{pmatrix}\\
    A_{-2}&=\frac{1}{2}\begin{pmatrix}
2 & 0 & \sqrt{-2}+1&\frac{1}{2}\sqrt{-2}\\
0 & 2 & \frac{1}{2}\sqrt{-2} & 1-\sqrt{-2}\\
1-\sqrt{-2} & -\frac{1}{2}\sqrt{-2} & 2 & 0\\
-\frac{1}{2}\sqrt{-2}& \sqrt{-2}+1 & 0 & 2\\
\end{pmatrix}
\end{align*}
\begin{align*}
B_{-2}&=\begin{pmatrix}
2 & 0 & 0 & 0 & 1 & 2 & 0 & 1 \\
0 & 4 & 0 & 0 & -2& 2 & -1& 0 \\
0 & 0 & 2 & 0 & 0 & 1 & 1 & -2 \\
0 & 0 & 0 & 4 & -1& 0 & 2 & 2 \\
1 & -2& 0 & -1& 2 & 0 & 0 & 0 \\
2 & 2 & 1 & 0 & 0 & 4 & 0 & 0 \\
0 & -1& 1 & 2 & 0 & 0 & 2 & 0 \\
1 & 0 & -2& 2 & 0 & 0 & 0 & 4 \\
\end{pmatrix}.
\end{align*}

Here, we give some examples of Hermitian lattices $L$ satisfying the conditions of Theorems \ref{Kodaira_dim_app}, \ref{Kodaira_dim_app2},  \ref{Kodaira_dim_app3} and  \ref{Kodaira_dim_app4}.
The definitions of the Hermitian lattices appearing in Proposition \ref{exist_app} are given there.

Below, we say \textit{a Hermitian lattice $L$ satisfies $(d,n,X)$} if $L$ is defined over $\OO_{\Q(\sqrt{d})}$, whose signature is $(1,n)$ and satisfies the condition of Theorem $X$.

\begin{prop}
\label{exist_app}
There exist Hermitian lattices satisfying the assumption of Theorems \ref{Kodaira_dim_app},  \ref{Kodaira_dim_app2},  \ref{Kodaira_dim_app3}, and  \ref{Kodaira_dim_app4}.
\end{prop}
\begin{proof}
It suffices to take $L$ as Table \ref{table:Hermitian_lattices_exist}.
\begin{center}
\renewcommand{\arraystretch}{1.5}
\begin{table}[h]
\caption{Hermitian lattices satisfying the assumptions.}
\label{table:Hermitian_lattices_exist}
    \begin{tabular}{|c||c|c|} \hline
     $(-1,5,\ref{Kodaira_dim_app})$&\begin{tabular}{c}
$L_{\mathbb{U}\oplus\mathbb{U}}\oplus L_{-1,\mathbb{E}_8}(-1)$, $L_{\mathbb{U}\oplus\mathbb{U}}\oplus L_{\mathbb{D}_8}(-1)$,\\ $ L_{\mathbb{U}\oplus\mathbb{U}}\oplus L_{-1,\mathbb{D}_4}(-1)\oplus L_{-1,\mathbb{D}_4}(-1)$ or Kond\=o's one \cite{Kondo3}\end{tabular}\\ \hline
      $(-1,5,\ref{Kodaira_dim_app2})$&$L_{-1,\mathbb{U}\oplus\mathbb{U}(2)}\oplus  L_{\mathbb{E}_8}(-2)$\\ \hline
      $(-2,5,\ref{Kodaira_dim_app})$&$L_{-2,\mathbb{U}\oplus\mathbb{U}(2)}\oplus L_{-2,\mathbb{D}_4}(-1)\oplus L_{-2,\mathbb{D}_4}(-1)$\\\hline
      $(-1,4,\ref{Kodaira_dim_app3})$&$L_{\mathbb{U}\oplus\mathbb{U}}\oplus L_{\mathbb{D}_6}(-1)$\\\hline
      $(-1,3,\ref{Kodaira_dim_app4})$&$L_{\mathbb{U}\oplus\mathbb{U}}(2)\oplus L_{-1,\mathbb{D}_4}$,  $L_{\mathbb{U}\oplus\mathbb{U}}(2)\oplus L_{-2,\mathbb{D}_4}$ or $L_{\mathbb{U}\oplus\mathbb{U}}\oplus L_{-1,\mathbb{D}_4}$\\\hline
$(-2,3,\ref{Kodaira_dim_app4})$&$L_{\mathbb{U}\oplus\mathbb{U}}(2)\oplus L_{-1,\mathbb{D}_4}$,  $L_{\mathbb{U}\oplus\mathbb{U}}(2)\oplus L_{-2,\mathbb{D}_4}$ or $ L_{\mathbb{U}\oplus\mathbb{U}(2)}\oplus L_{-2,\mathbb{D}_4}$\\\hline
    \end{tabular}
    \end{table}
  \renewcommand{\arraystretch}{1.0}
  \end{center}
\end{proof}

\end{document}